\documentclass[11pt,a4paper,leqno]{amsart}
\usepackage[T1]{fontenc}
\usepackage[utf8]{inputenc}
\usepackage{lmodern}
\usepackage[a4paper,margin=28mm]{geometry}
\usepackage{amsmath,amssymb,amsthm,mathrsfs,bm}
\usepackage{graphicx,booktabs,multirow,float}
\usepackage{caption,adjustbox}
\usepackage{tikz}
\usetikzlibrary{patterns}
\usepackage{pgfplots}
\pgfplotsset{compat=1.18}
\usepackage[numbers,sort&compress]{natbib}
\usepackage[expansion=false]{microtype}
\usepackage[hidelinks,unicode]{hyperref}
\hypersetup{
  pdftitle={Further results on [k]-Roman domination on cylindrical grids C(m) x P(n)},
  pdfauthor={Simon Brezovnik and Janez Zerovnik}
}
\newtheorem{theorem}{Theorem}

\newtheorem{lemma}[theorem]{Lemma}
\newtheorem{corollary}[theorem]{Corollary}

\theoremstyle{definition}

\title[{Further results on $[k]$-Roman domination}]{Further results on $[k]$-Roman domination on cylindrical grids $C_m\Box P_n$}
\author[S. Brezovnik]{Simon Brezovnik}
\address{Faculty of Mechanical Engineering, University of Ljubljana, Ljubljana, Slovenia; Institute of Mathematics, Physics and Mechanics, Ljubljana, Slovenia; Faculty of Mathematics and Physics, University of Ljubljana, Ljubljana, Slovenia}
\email{simon.brezovnik@fs.uni-lj.si}
\author[J. \v{Z}erovnik]{Janez \v{Z}erovnik}
\address{Faculty of Mechanical Engineering, University of Ljubljana, Ljubljana, Slovenia; Rudolfovo -- Science and Technology Centre, Novo Mesto, Slovenia}
\email{janez.zerovnik@fs.uni-lj.si}
\thanks{Corresponding author: Simon Brezovnik. The authors contributed equally to this work.}
\date{}

\begin{document}
\raggedbottom
\begin{abstract}
{In this paper, we study the $[k]$-Roman domination number of cylindrical
graphs $C_m \Box P_n$. Our analysis begins with a general lower bound based on
local neighborhood constraints. We show that
$\gamma_{[kR]}(C_m\Box P_n) > \frac{(k+1)mn}{5}.$
By exploiting the connection between $[k]$-Roman domination and efficient
domination, we characterize the cylindrical graphs for which the extremal
local configuration \(f(N[x])=k+1\) for every \(x\in V(C_m\Box P_n)\) can
occur. We show that this happens precisely for so called efficient graphs. For fixed
small values $m\in\{5,\ldots,8\}$, we construct explicit periodic
$[k]$-Roman dominating functions that yield sparse upper bounds. These
constructions are complemented by a general uniform upper bound and by
packing-refined bounds. A systematic comparison of the resulting bounds shows
how their relative strength depends on the parameter $k$ and on the length of
the path.}
\end{abstract}
\maketitle

\section{Introduction}

\smallskip
\noindent
Roman-type domination parameters model defensive resource allocation in networks by allowing vertices to carry different levels of protection. 
The classical \emph{Roman domination} concept, introduced by Cockayne et al.~\cite{Cockayne1985}, assigns to each vertex a label from $\{0,1,2\}$ so that every vertex labeled $0$ has a neighbor labeled $2$. 
{Label 2 corresponds to the case when two legions are positioned at a site, and in general the label is   the weight of a vertex.}
The minimum total weight of such a labeling defines the Roman domination number $\gamma_R(G)$. 
This model reflects the idea that strongly defended vertices can protect adjacent undefended regions, thereby reducing the total amount of resources needed.

\smallskip
\noindent
Since its introduction, Roman domination has become a central topic within domination theory and has been investigated for numerous graph classes; see the surveys~\cite{Chellali2021VarietiesI,Chellali2020VarietiesII,Chellali2024VarietiesIII,Chellali2025VarietiesIV,math11020351}. 
Stronger variants were later proposed to model more robust protection schemes. 
In particular, double Roman domination~\cite{Beeler2016} allows labels from $\{0,1,2,3\}$ with reinforced local conditions, while triple and higher-order versions were developed as natural extensions~\cite{ABDOLLAHZADEHAHANGAR2021125444,doi:10.1142/S1793830921501305}. 
These developments revealed a hierarchical structure among Roman-type parameters.

\smallskip
\noindent
A unified framework was introduced in~\cite{ABDOLLAHZADEHAHANGAR2021125444} under the name \emph{$[k]$-Roman domination}. Recent work has addressed $[k]$-Roman domination on various graph families. 
General bounds and exact values for paths, cycles, complete graphs, stars, and selected trees were obtained in~\cite{Khalili2023}, while structural and complexity aspects were studied in~\cite{Valenzuela2024}. 
Roman-type parameters in Cartesian product graphs, particularly grid and cylindrical graphs, have attracted attention as well; see, for instance,~\cite{AnuAparna2024} for double Roman domination on products of paths and cycles.

\smallskip
\noindent
Cartesian products of cycles and paths, that is graphs of the form $C_m \Box P_n$, are a fundamental class of cylindrical grids. 
Their regular structure makes them particularly suitable for constructive domination arguments, while still exhibiting rich combinatorial behavior. 
In the classical Roman setting, such graphs already display nontrivial patterns depending on the parity and size of the parameters.

\smallskip
\noindent
In this paper, we investigate the $[k]$-Roman domination number of cylindrical
graphs $C_m \Box P_n$. {
We first revisit the general density-type lower bound and strengthen it by
showing that, in fact,
\[
\gamma_{[kR]}(C_m\Box P_n)> \frac{(k+1)mn}{5}.
\]
Moreover, using the connection between $[k]$-Roman domination and efficient
domination, we characterize precisely when the lowest possible local
neighborhood weight can be attained, that is, when
\[
f(N[x])=k+1 \quad \text{for every } x\in V(C_m\Box P_n).
\]
We show that this occurs exactly for  efficient graphs , namely for
$C_m\Box P_1$ with $m\equiv0\pmod 3$ and for $C_m\Box P_2$ with
$m\equiv0\pmod 4$. Hence, for all other cylindrical graphs, the local
lower-bound configuration is impossible.
}

\smallskip
\noindent
We then construct explicit periodic $[k]$-Roman dominating functions that yield upper bounds for fixed small values $m\in\{5,\ldots,8\}$. 
For each case we refine the basic constructions by means of ceiling-type adjustments and packing-based reductions, and we systematically compare the resulting bounds. 
This analysis reveals how the relative strength of different constructions depends on $k$ and on the length of the path.

\section{Preliminaries}

\smallskip
\noindent
Let $G=(V,E)$ be a finite, simple graph and let $k\ge 1$. For a vertex $v\in V(G)$, we denote by $N(v)$ its (open) neighborhood and by 
$N[v]=N(v)\cup\{v\}$ its closed neighborhood.

\smallskip
\noindent
A function $f:V\to\{0,1,\ldots,k+1\}$ is a \emph{$[k]$-Roman dominating function}
($[k]$-RDF) if for every vertex $v$ with $f(v)<k$,
\[
f(N[v]) \ge k + |AN(v)|,
\quad\text{where}\quad
AN(v)=\{u\in N(v): f(u)>0\}.
\]

\noindent
{The set $AN(v)$ is  called the active neighborhood of $v$.}
Intuitively, the condition requires that vertices of small value must be supported by sufficiently large total weight in their closed neighborhood, with an additional dependence on the number of positively weighted neighbors.

\smallskip
\noindent
The \emph{weight} of $f$ is $w(f)=\sum_{v\in V} f(v)$, and the \emph{$[k]$-Roman
domination number} is
\[
{\gamma_{[kR]}(G)}=\min\{w(f): f\text{ is a $[k]$-RDF on }G\}.
\]

\noindent
A $[k]$-Roman dominating function $f$ on $G$ is called a \emph{${\gamma_{[kR]}}$-function}
if it attains the minimum weight, that is, if $w(f)={\gamma_{[kR]}(G)}$.

\noindent
We use $P_n$ for the path on vertices $\{0,1,\ldots,n-1\}$ and $C_m$ for the cycle
on $\{0,1,\ldots,m-1\}$ (indices modulo $m$).
The Cartesian product $C_m\Box P_n$ has vertex set $V(C_m)\times V(P_n)$; vertices
$(i,j)$ and $(i',j')$ are adjacent if either $i=i'$ and $jj'\in E(P_n)$, or
$j=j'$ and $ii'\in E(C_m)$.

\smallskip
\noindent
Note that in the cylindrical graph $C_m\Box P_n$, every interior vertex has degree $4$, 
and hence, its closed neighborhood contains at most five vertices.

\smallskip
\noindent
A set $D \subseteq V(G)$ is an \emph{efficient dominating set} if for every vertex
$v \in V(G)$ there exists a unique vertex $u \in D$ such that
$v \in N[u]$. A graph is said to be \emph{efficient} if it contains an efficient dominating
set.

\smallskip
\noindent
A set $D \subseteq V(G)$ is called a \emph{packing} if for every two distinct vertices 
$u,v\in D$ we have $N[u]\cap N[v]=\emptyset$. 
The maximum cardinality of a packing in $G$ is called the \emph{packing number} 
of $G$ and is denoted by $\rho(G)$.

\section{\texorpdfstring{Lower bound for ${\gamma_{[kR]}}(C_m \Box P_n)$}{Lower bound for cylindrical grids}}

\smallskip
\noindent
{
In this section, we derive a lower bound for
$\gamma_{[kR]}(C_m \Box P_n)$ by studying configurations in which
each closed neighborhood has total weight exactly $k+1$.
This represents the optimal local scenario for $[k]$-Roman domination,
since no smaller neighborhood sum can guarantee domination. By characterizing
when such extremal local configurations can occur, we obtain a general lower
bound and identify the cases in which this local optimum is attainable.
}

\smallskip
\noindent
{
We first prove a simple local lemma that will be used in the proof of
Theorem~\ref{thm:local-sum-efficient}. It shows that, under the extremal
assumption $f(N[x])=k+1$ for every vertex $x$, no vertex can have more than
one active neighbor.
}

\begin{lemma}\label{lem:at-most-one-active-neighbor}
{
Let $f$ be a $[k]$-Roman dominating function on a graph $G$ such that
\[
f(N[x])=k+1
\quad \text{for every } x\in V(G).
\]
Then every vertex of $G$ has at most one active neighbor.
}
\end{lemma}

\begin{proof}
{Let $x\in V(G)$. Since $f$ is a $[k]$-Roman dominating function it holds \[
f(N[v]) \ge k + |AN(v)|\] and by assumption, we have $f(N[x])=k+1$. It follows that
\[
k+1=f(N[x])\ge k+|AN(x)|.
\]
Hence $|AN(x)|\le 1$.}
\end{proof}

{{\bf Remark.}
Considering  the proof of Lemma \ref{lem:at-most-one-active-neighbor}
we further observe that 
  if $f(x) = k+1$, then $  f(N(x)) =0 $ and  $|AN(x)|=0$. 
On the other hand,  if  $f(x) = a < k+1$,  then  $  f(N(x)) =  k+1 -a  = b $.
Note that the weight $b$ must be concentrated at exactly one neighbor of $x$, as otherwise
we would have 
$ f(N[x]) = k+1 <  k+|AN(x)|$ contradicting the assumption.
}

\smallskip
\noindent
{
We now characterize the cylindrical grids for which this extremal local
configuration can occur. The following theorem shows that such a configuration
is possible precisely in the cases where the graph admits an efficient
dominating set.
}

\begin{theorem}\label{thm:local-sum-efficient}
Let $G=C_m\Box P_n$.
The following statements are equivalent:
\begin{enumerate}
\item $G$ admits an efficient dominating set.
\item There exists a $[k]$-Roman dominating function
$f:V(G)\to\{0,1,\dots,k+1\}$ such that
\[
\sum_{u\in N[v]} f(u)=k+1
\quad \text{for every } v\in V(G).
\]
\item $G$ is either $C_m\Box P_1$ with $m\equiv 0 \pmod{3}$,
or $C_m\Box P_2$ with $m\equiv 0 \pmod{4}$.
\end{enumerate}
\end{theorem}

\begin{proof}

\noindent
(1)$\Rightarrow$(2).
Let $D$ be an efficient dominating set of $G$. Define $f$ by
$f(v)=k+1$ for $v\in D$ and $f(v)=0$ otherwise. Since every vertex
belongs to the closed neighborhood of exactly one vertex of $D$, we have
$\sum_{u\in N[v]} f(u)=k+1$ for all $v\in V(G)$.

\medskip
\noindent
(2)$\Rightarrow$(1).
{Assume there exists a $[k]$-Roman dominating function
$f:V(G)\to\{0,1,\dots,k+1\}$ such that
\[
\sum_{u\in N[v]} f(u)=k+1
\quad \text{for every } v\in V(G).
\] 
We first claim that no vertex receives weight strictly smaller than $k+1$. 
In other words,  we have  either $f(v) =0$ or $f(v) =k+1$.
Suppose, to the contrary,
that for some vertex $v$ we have  $f(v)=a$  and $0<a<k+1$ . Since $f(N[v])=k+1$, exactly one neighbor
of $v$ has positive  weight $b=k+1-a$, and all other neighbors of $v$ have weight $0$.
Let $y$ be the unique neighbor of $v$ with $f(y)=b >0$. Since $G$ has no leaves,
there exists another neighbor $z$ of $y$. 
Recall that any vertex may have at most one active neighbor by Lemma \ref{lem:at-most-one-active-neighbor}.
As the active neighbor of $y$ is $v$, we must have $f(z)=0$.
Furthermore, all  other  neighbors of $z$ must have weight 0, because the 
only active neighbor of $z$ is $y$. Hence
\[
f(N[z])=f(y)=b<k+1 ,
\]
which contradicts the assumption $f(N[z])=k+1$. Thus no vertex receives
weight $<k+1$.
}

\smallskip
\noindent
{
Since $f(N[x])=k+1$ for every vertex $x$, each closed neighborhood contains
at least one vertex of positive weight. By the previous paragraph, it contains
exactly one such vertex. Moreover, the weight of this unique positive vertex
must be equal to $k+1$.
}

\smallskip
\noindent
Consequently, the set
\[
D=\{v\in V(G): f(v)=k+1\}
\]
satisfies $|N[v]\cap D|=1$ for all $v$, and hence $D$ is an efficient dominating set.

\medskip
\noindent
(1)$\Leftrightarrow$(3).
By \cite[Theorem~2]{BrezovnikZerovnik2026_1}, the graph $C_m\Box P_n$
admits an efficient dominating set if and only if either $n=1$ and
$m\equiv 0 \pmod{3}$, or $n=2$ and $m\equiv 0 \pmod{4}$.
\end{proof}

\noindent
The previous theorem shows that configurations in which every closed
neighborhood has total weight exactly $k+1$ correspond precisely to
efficient domination. In particular, such configurations represent the
most economical local arrangement of weights.

\smallskip
\noindent
{We now explain how the preceding characterization strengthens the general
lower bound from~\cite{BrezovnikZerovnik2026_1}. In that paper, it was shown that for cylindrical
graphs one has
\[
\gamma_{[kR]}(C_m\Box P_n)\ge \frac{(k+1)mn}{5}.
\]
The characterization proved above allows us to determine whether equality in
this density-type bound can occur. Indeed, equality would require the
corresponding labeling to realize the extremal local situation in which every
closed neighborhood has total weight exactly \(k+1\). By
Theorem~\ref{thm:local-sum-efficient}, such a configuration can occur only
when \(C_m\Box P_n\) admits an efficient dominating set. We now show that even
in these exceptional cases the value of \(\gamma_{[kR]}\) is strictly larger
than the density bound.
}

\begin{corollary}\label{cor:strict-lower-bound}
Let $G=C_m\Box P_n$, where $m\ge 3$ and $n\ge 1$. Then
\[
\gamma_{[kR]}(G)>\frac{(k+1)mn}{5}.
\]
\end{corollary}

\begin{proof}
{
By~\cite{BrezovnikZerovnik2026_1}, we have
\[
\gamma_{[kR]}(G)\ge \frac{(k+1)mn}{5}.
\]
Suppose, to the contrary, that equality holds. Then the lower-bound
configuration must be attained, and hence every closed neighborhood has total
weight exactly \(k+1\). By Theorem~\ref{thm:local-sum-efficient}, this implies
that \(G\) admits an efficient dominating set. Therefore either
\(n=1\) and \(m\equiv0\pmod 3\), or \(n=2\) and \(m\equiv0\pmod 4\).
}

\smallskip
\noindent
{
If \(n=1\), then \(G\cong C_m\). Since \(m\equiv0\pmod 3\), the graph \(C_m\)
has an efficient dominating set of size \(m/3\), and hence
\[
\gamma_{[kR]}(C_m)=(k+1)\frac{m}{3}.
\]
However,
\[
(k+1)\frac{m}{3}>\frac{(k+1)m}{5}
=\frac{(k+1)mn}{5},
\]
which contradicts equality in the density bound.
}

\smallskip
\noindent
{
If \(n=2\), then \(m\equiv0\pmod 4\). In this case
\(C_m\Box P_2\) has an efficient dominating set of size \(m/2\), and therefore
\[
\gamma_{[kR]}(C_m\Box P_2)=(k+1)\frac{m}{2}.
\]
But
\[
(k+1)\frac{m}{2}>\frac{2m(k+1)}{5}
=\frac{(k+1)mn}{5},
\]
again contradicting equality in the density bound. Thus equality is impossible,
and the strict inequality follows.
}
\end{proof}

\smallskip
\noindent
{
\smallskip
\noindent
{
Corollary~\ref{cor:strict-lower-bound}
can be also stated in the rounded form.

\begin{corollary}\label{cor:rounded-lower-bound}
Let $G=C_m\Box P_n$, where $m\ge 3$ and $n\ge 1$. Then
\[
\gamma_{[kR]}(G)
\ge
\left\lfloor(k+1)\frac{mn}{5}\right\rfloor+1 .
\]
\end{corollary}
}
}

\section{\texorpdfstring{Upper bounds for ${\gamma_{[kR]}}(C_m\Box P_n)$}{Upper bounds for cylindrical grids}}
\label{sec:upper-bounds}

\smallskip
\noindent
{
In this section we derive constructive upper bounds for the $[k]$-Roman
domination number of cylindrical graphs \(C_m\Box P_n\) for
\(m\in\{5,6,7,8\}\).
We use three types of constructions. The first type consists of sparse
periodic labelings that are adapted to the value of \(m\). The second type is
a uniform local-density construction using the base weight
\(\left\lceil (k+4)/5\right\rceil\). The third type is a packing-refined
construction using the base weight \(\left\lceil (k+5)/5\right\rceil\) and then
decreasing the weight on a maximum packing set.
}

\subsection{Sparse periodic constructions}

\smallskip
\noindent
{
We begin with the sparse periodic constructions. In each case, the labeling is
defined by a periodic pattern along the \(P_n\)-direction and is adjusted only
near the boundary fibres. Since the constructions share the same general
principle, we state the resulting bounds in a single theorem.
}

{
More precisely, all the constructions will be presented as matrices, 
consisting of three (possible empty) parts: 
(1) the initial columns, usually just one (or even none),
(2) the pattern that repeats periodically, and 
(3) the tail. 
The length of the tail is  $t= (n-1-i) \bmod p$,
where $i$ is the number of initial columns and $p$ is the period.
 In each matrix, the columns correspond to fibres of the path and the rows correspond to vertices of the cycle. Boundary fibres are
then corrected locally. 
The adjustments will affect only tail and initial columns.  

{\setcounter{MaxMatrixCols}{12}
\begin{equation}\label{SplosnaShema}
\begin{pmatrix}
\multicolumn{3}{c}{\textrm{first col(s)}}  &  | &\multicolumn{3}{c}{\textrm{the pattern}}  &  | &\multicolumn{3}{c}{\textrm{tail}} \\
 0   & \cdots &   0   & | &0  &\cdots   & 0  & | &0  &\cdots  &0  \\  
      \vdots &      & \vdots &|     & \vdots &    &  \vdots& | &     \vdots &  & \vdots    \\  
 0   & \cdots &   0   & | &0  &\cdots   & 0  & | &0  &\cdots  &0  
\end{pmatrix}.
\end{equation}
}
}

\begin{theorem}\label{thm:sparse-periodic-bounds}
{
Let \(n\ge2\). The following upper bounds hold.
\begin{align}
\gamma_{[kR]}(C_{5}\Box P_n)
&\le n(k+1)+2k,
\label{eq:ub-C5Pn-linear}\\[4pt]
\gamma_{[kR]}(C_6\Box P_n)
&\le
\left\lceil \frac{4n}{3}\right\rceil(k+1)
+
\begin{cases}
k+1, & \text{if } n\equiv 0 \pmod{3},\\
0,   & \text{if } n\equiv 1 \pmod{3},\\
k,   & \text{if } n\equiv 2 \pmod{3},
\end{cases}
\label{eq:ub-C6Pn-linear}\\[4pt]
\gamma_{[kR]}(C_{7}\Box P_n)
&\le
\begin{cases}
(n+1)(k+1)+\dfrac{n-1}{2}(k+1)+2k, & n\ \text{odd},\\[8pt]
n(k+1)+\dfrac{n}{2}(k+1)+2k+1, & n\ \text{even},
\end{cases}
\label{eq:ub-C7Pn-linear}\\[4pt]
\gamma_{[kR]}(C_8\Box P_n)
&\le
2n(k+1)
-
\left(
\left\lfloor\frac{n-2}{5}\right\rfloor
+
\left\lfloor\frac{n}{5}\right\rfloor
\right)
+2k.
\label{eq:ub-C8Pn-linear}
\end{align}
}
\end{theorem}

\begin{proof}
{
All four bounds are obtained by explicit periodic labelings along the
\(P_n\)-direction. The matrices below describe one period or the repeating
part of the construction.
A direct neighborhood verification
shows that every vertex \(v\) with \(f(v)<k\) satisfies
\[
f(N[v])\ge k+|AN(v)|.
\]
}

\smallskip
\noindent
{\textit{The case \(m=5\).}}
{
We use the a pattern  of length five, and one (adjusted) initial column. 
In the tail, only the last column is adjusted. 

More precisely, if  \(f(j,1)=k+1\), then on the
initial fibre \(F_0\) we set
\[
f(j+1,0)=k,\qquad f(j-2,0)=k+1,
\]
with indices modulo \(5\), and all remaining vertices of \(F_0\) receive label
\(0\). Similarly, if \(f(j,n-2)=k+1\), then on the final fibre \(F_{n-1}\) we
set
\[
f(j-1,n-1)=k,\qquad f(j+2,n-1)=k+1,
\]
and all remaining vertices of \(F_{n-1}\) receive label \(0\). The total weight
is \(n(k+1)+2k\), giving \eqref{eq:ub-C5Pn-linear}.

For example, when  $n = 5k +2$, the construction is given below (the pattern repeats $k$ times),
{\setcounter{MaxMatrixCols}{11}
\begin{equation}\label{pat:m0-period}
\begin{pmatrix}
  0   & | & k+1 & 0   & 0   & 0   & 0   & | & k+1   \\
  {\bf k}   & | & 0   & 0   & 0   & k+1 & 0   & | & 0 \\
  0   & | & 0   & k+1 & 0   & 0   & 0   & | & {\bf k}   \\
  k+1 & | & 0   & 0   & 0   & 0   & k+1 & | & 0 \\
  0   & | & 0   & 0   & k+1 & 0   & 0   & | & 0 
\end{pmatrix}.
\end{equation}
}
and the special case $n=6$ is provided in the  next example. 
{\setcounter{MaxMatrixCols}{11}
\begin{equation}\label{pat:m0-perio6d}
\begin{pmatrix}
  0    & k+1 & 0   & 0   & 0   & {\bf k}   \\
  {\bf k}    & 0   & 0   & 0   & k+1 & 0     \\
 0    & 0   & k+1 & 0   & 0   & 0    \\
  k+1 & 0   & 0   & 0   & 0   & k+1   \\
  0   & 0   & 0   & k+1 & 0   & 0   
\end{pmatrix}.
\end{equation}
}
}

\smallskip
\noindent
{\textit{The case \(m=6\).}}
{
We use the following periodic pattern:
}

{
\begin{equation}\label{pat:m6-period}
\begin{pmatrix}
 k+1 & 0   & 0   & 0   & 0   & 0 & |& *   \\
0   & 0   & 0   & k+1 & 0   & 0 & | &*   \\
0   & k+1 & 0   & 0   & 0   & k+1 & | & *\\
0   & 0   & 0   & k+1 & 0   & 0 & | & * \\
k+1 & 0   & 0   & 0   & 0   & 0 & | & * \\
0   & 0   & k+1 & 0   & k+1 & 0 & | & * \\
\end{pmatrix}.
\end{equation}
The asterisks in the pattern above represent the values of the \([k]\)-Roman dominating function on the vertices of the tail. Note that in this case, there is no need to add initial column(s) at the left side since all vertices in the 
first column of the pattern are dominated. 
In the tail there are up to six columns, more precisely  $t= 1 + (n-1) \bmod 6  $.  
}

\smallskip
\noindent
{
The weight of the repeated part is
\(\left\lceil 4n/3\right\rceil(k+1)\). The final fibre is corrected according
to \(n\bmod 3\)
{ (or, equivalently, \(t \bmod 3\) )}
\smallskip
\noindent
If $n\equiv 0\pmod 3$, then set
\[
f(2,n-1)=k+1
\qquad\text{and}\qquad
f(5,n-1)=k+1,
\]
and label all other vertices of $F_{n-1}$ by $0$.
This adds $(k+1)$ to the total sum of inner weights.

\smallskip
\noindent
If $n\equiv 1\pmod 3$, then no correction is needed on the last fibre.

\smallskip
\noindent
If $n\equiv 2\pmod 3$, then let $j\in\{0,1,\dots,5\}$ be such index that
$f(j,n-2)=k+1$ and $f(j+2,n-2)=k+1$ (indices taken modulo $6$).
Set
\[
f(j+1,n-1)=k
\qquad\text{and}\qquad
f(j-2,n-1)=k+1,
\]
and label all other vertices of $F_{n-1}$ by $0$.
This adds $k$ to the total sum of inner weights.  This gives  \eqref{eq:ub-C6Pn-linear}.
}

\smallskip
\noindent

{ {\textit{The case \(m=7\).}} The initial column and  the pattern of length 14 is   

{\setcounter{MaxMatrixCols}{18}\footnotesize
\setlength{\arraycolsep}{2.5pt}
\begin{equation}\label{pat:m2}
\begin{pmatrix}
k+1 & | & 0 & 0 & 0 & k+1 & 0 &  0 & 0 & 0 &  k+1 & 0 &  0 &  0 & 0 & k+1 & |& *  \\
0 & | & 0 & k+1 & 0 & 0 & 0 & 0 & k+1 & 0 & 0 & 0 & 0 &  k+1 & 0  & 0& |& *   \\
k+1 & | & 0 & 0 & 0 & 0 & k+1 &  0 & 0 & 0 & 0 & k+1 &   0 & 0& 0 & k+1 & |& *  \\
0 & | & 0 & 0 & k+1 & 0 & 0 &  0 & 0 &  k+1 & 0 & 0 &  0 & k+1& 0 & 0& |& *   \\
0 & | & k+1 & 0 & 0 & 0 & 0 &  k+1 & 0 & 0 & 0 & k+1 & 0 & 0 & 0 & 0 & |& *   \\
{\bf k} & | & 0 & 0 & 0 & k+1 & 0 &0 & 0 & k+1 & 0 & 0 & 0 &   0 & k+1  & 0& |& *   \\
0& | & 0 & k+1 & 0 & 0 & 0 & k+1 & 0 & 0 & 0 & 0& k+1 &0 &0 & 0 & |& *  
\end{pmatrix}.
\end{equation}
}
}

\smallskip
\noindent
{
More precisely, each odd column contains exactly two vertices of weight $k+1$,
which occur in rows $r$ and $r+2$ for some index $r$ taken modulo $7$. From one
odd column to the next odd column, these two positions move one row upward.
Each even column contains exactly one vertex of weight $k+1$, and this entry
also moves one row upward from one even column to the next, with all indices
taken modulo $7$.
}

\smallskip
\noindent
Since the shift is by one row on a $7$-cycle and the odd--even structure has
period $2$, the entire pattern repeats after $14$ columns along the
$P_n$-direction.

\smallskip
\noindent
It remains to define the labeling on the boundary fibre $F_{n-1}$.
(All other fibres in the tail follow the pattern).
Assume that $n$ is even. Let $j\in\{0,1,\dots,6\}$ be such that $f(j,n-2)=f(j+2,n-2)=k+1.$
Then we set
\[
f(j,n-1)=k+1
\qquad\text{and}\qquad
f(j-3,n-1)=k+1,
\]
where the index $j-3$ is taken modulo $7$, and all other vertices of $F_{n-1}$
are labeled $0$.

\smallskip
\noindent
Assume now that $n$ is odd. Let $j\in\{0,1,\dots,6\}$ be such that $f(j,n-2)=k+1.$
Then we set
\[
f(j-1,n-1)=k
\qquad\text{and}\qquad
f(j-3,n-1)=f(j+2,n-1)=k+1,
\]
where all indices are taken modulo $7$, and all other vertices of $F_{n-1}$ are
labeled $0$.  This yields \eqref{eq:ub-C7Pn-linear}.

\smallskip
\noindent
{\textit{The case \(m=8\).}}
{
We use the  following initial column and periodic pattern:

\begin{equation}\label{eq:pat-m8-new}
\begin{pmatrix}
k+1    &| & k+1   & 0   & k & 0   & 0   &|&  *\\
0  &|& 0   & 0   & 0   & k+1   & 0 & |&  
*\\
0   &|  & 0   & k+1 & 0   & 0   & 0  &|&   *\\
k+1   &|  & 0   & 0   & 0   & 0 & k+1   &|&   *\\
0    &| & 0 & 0   & k+1   & 0   & 0 &|&   *\\
0   &|  & k+1   & 0   & 0 & 0   & 0 &|&   *\\
{\bf k}  &|  & 0   & 0   & 0   & k+1   & 0 &|&  *\\
0   &|  & 0   & k+1   & 0   & 0   & k   &|&   *
\end{pmatrix}.
\end{equation}
}

\smallskip
\noindent
{
In this pattern, some entries of weight \(k+1\) are replaced by entries of
weight \(k\). The number of such replacements is
\[S=
\left\lfloor\frac{n-2}{5}\right\rfloor
+
\left\lfloor\frac{n}{5}\right\rfloor .
\]
Each replacement decreases the total weight by exactly $1$. Therefore the weight
of the interior part is
\[
(2(n-2)-S)(k+1)+Sk,
\]
since $2(n-2)-S$ vertices retain the value $k+1$ and $S$ vertices receive
the value $k$.
}

\smallskip
\noindent
Finally,
\[
(2(n-2)-S)(k+1)+Sk
=
2(n-2)(k+1)-S(k+1)+Sk
=
2(n-2)(k+1)-S.
\]

\smallskip
\noindent
It remains only to correct the tail, more precisely, the final fibre. 
Depending on the residue class of \(n\) modulo the period, the last interior
fibre may contain either two entries of weight \(k+1\), or one entry of weight
\(k+1\) together with one entry of weight \(k\). In both cases, the final
boundary fibre can be chosen with two entries of weight \(k+1\) and one entry
of weight \(k\). Since the verification is purely local and
analogous in all cases, we leave the straightforward case check to the reader. Hence the two boundary fibres together contribute
\[
4(k+1)+2k.
\]

\smallskip
\noindent
Adding the boundary correction, which contributes $4(k+1)+2k$, we obtain
\[
w(f)
=
2(n-2)(k+1)
-
\left(
\left\lfloor\frac{n-2}{5}\right\rfloor
+
\left\lfloor\frac{n}{5}\right\rfloor
\right)
+4(k+1)+2k,
\]
which simplifies to
\[
2n(k+1)
-
\left(
\left\lfloor\frac{n-2}{5}\right\rfloor
+
\left\lfloor\frac{n}{5}\right\rfloor
\right)
+2k.
\]
\end{proof}

\smallskip
\noindent
{
We illustrate the sparse periodic constructions with one example for each
\(m\in\{5,6,7,8\}\):
\small
\[
\begin{array}{cc}
\begin{array}{c}
C_5\Box P_4 \\[4pt]
\begin{pmatrix}
0 & k+1 & 0 & 0 \\
k & 0 & 0 & k \\
0 & 0 & k+1 & 0 \\
k+1 & 0 & 0 & 0 \\
0 & 0 & 0 & k+1
\end{pmatrix}
\end{array}
&
\begin{array}{c}
C_6\Box P_4 \\[4pt]
\begin{pmatrix}
k+1 & 0   & 0   & 0 \\
0   & 0   & 0   & k+1 \\
0   & k+1 & 0   & 0 \\
0   & 0   & 0   & k+1 \\
k+1 & 0   & 0   & 0 \\
0   & 0   & k+1 & 0
\end{pmatrix}
\end{array}
\\[24pt]
\begin{array}{c}
C_7\Box P_4 \\[4pt]
\begin{pmatrix}
k+1 & 0   & 0   & 0 \\
0   & 0   & k+1 & 0 \\
k+1 & 0   & 0   & 0 \\
0   & 0   & 0   & k+1 \\
0   & k+1 & 0   & 0 \\
k   & 0   & 0   & 0 \\
0   & 0   & k+1 & k+1
\end{pmatrix}
\end{array}
&
\begin{array}{c}
C_8\Box P_5 \\[4pt]
\begin{pmatrix}
k+1 & k+1 & 0 & k   & 0    \\
0   & 0   & 0 & 0   & k+1 \\
0   & 0 & k+1 & 0   & 0   \\
k+1 & 0  & 0 & 0   & k    \\
0   & 0  & 0 & k+1 & 0   \\
0   & k+1  & 0 & 0   & 0   \\
k   & 0  & 0 & 0   & k+1 \\
0   & 0 & k+1 & 0   & 0  
\end{pmatrix}
\end{array}
\end{array}
\]
These examples illustrate the periodic patterns used in the proof. The general
labelings are obtained by repeating the same patterns along the
\(P_n\)-direction and applying the corresponding boundary corrections.
}

\subsection{Uniform local-density constructions}

\smallskip
\noindent
{
We next state the uniform local-density bounds. These constructions assign a
constant weight to every vertex in the interior fibres and a suitable constant
weight to every vertex in the boundary fibres.
} 

\begin{theorem}\label{thm:uniform-upper-bounds}
{
Let \(n\ge4\). Then
\begin{equation}\label{eq:uniform-upper-general}
\gamma_{[kR]}(C_m \Box P_n)
\le
m (n-2)\left\lceil\frac{k+4}{5}\right\rceil
+
2m\left\lceil
\frac{k+3-\left\lceil\frac{k+4}{5}\right\rceil}{3}
\right\rceil
\le
\frac{m(3nk+27n+2k-2)}{15}.
\end{equation}
}
\end{theorem}

\begin{proof}
{
Let
\[
a=\left\lceil\frac{k+4}{5}\right\rceil,
\qquad
b=\left\lceil\frac{k+3-a}{3}\right\rceil .
\]
Assign weight \(a\) to each vertex in the interior
fibres \(F_1,\ldots,F_{n-2}\), and weight \(b\) to each vertex in the two
boundary fibres \(F_0\) and \(F_{n-1}\).
}

\smallskip
\noindent
{
If \(v\) is an interior vertex, then its closed neighborhood contains five
vertices. Hence
\[
f(N[v])\ge 5a\ge k+4.
\]
Since an interior vertex has at most four active neighbors, we have
\[
f(N[v])\ge k+|AN(v)|,
\] as needed.
}

\smallskip
\noindent
{
If \(v\) lies in a boundary fibre, then its closed neighborhood receives one
contribution \(a\) from the adjacent interior fibre and three contributions
\(b\) from the boundary fibre. Therefore
\[
f(N[v])\ge a+3b\ge k+3.
\]
Since a boundary vertex has at most three active neighbors, it follows that
\[
f(N[v])\ge k+|AN(v)|.
\]
Thus the construction {provides}  a \( [k]\)-Roman dominating function.
}

\smallskip
\noindent
{
The total weight of this labeling is
\[
m(n-2)a+2mb,
\]
and hence
\[
\gamma_{[kR]}(C_m\Box P_n)
\le
m(n-2)\left\lceil\frac{k+4}{5}\right\rceil
+
2m\left\lceil
\frac{k+3-\left\lceil\frac{k+4}{5}\right\rceil}{3}
\right\rceil .
\]
}

\smallskip
\noindent
{
It remains to derive the ceiling-free estimate. Using
\[
\left\lceil\frac{k+4}{5}\right\rceil
\le
\frac{k+4}{5}+1
=
\frac{k+9}{5}
\]
and, since \(a=\left\lceil\frac{k+4}{5}\right\rceil\ge \frac{k+4}{5}\),
\[
b=
\left\lceil\frac{k+3-a}{3}\right\rceil
\le
\frac{k+3-\frac{k+4}{5}}{3}+1
=
\frac{4k+26}{15},
\]
we obtain
\[
m(n-2)a+2mb
\le
m(n-2)\frac{k+9}{5}
+
2m\frac{4k+26}{15}.
\]
Therefore
\[
m(n-2)a+2mb
\le
\frac{m(3nk+27n+2k-2)}{15}.
\]
This proves the claimed bound.
}
\end{proof}

\phantomsection\label{C5Cn}
\phantomsection\label{C}
\phantomsection\label{thm:C7Pn_uniform}
\phantomsection\label{thm:C8Pn_uniform}

\subsection{\texorpdfstring{Packing numbers of $C_{m}\Box P_n, m \in \{5, 6, 7, 8\}$}{Packing numbers for m = 5, 6, 7, 8}}

\smallskip
\noindent
{
The third family of constructions uses packing sets. We first collect the
packing numbers needed for \(m=5,6,7,8\). The figures display periodic packing
patterns attaining the stated values.
}

\begin{lemma}\label{lem:packing-summary}
{
For \(n\ge4\), the packing numbers of \(C_m\Box P_n\), where
\(m\in\{5,6,7,8\}\), are given by
\[
\rho(C_m\Box P_n)=
\begin{cases}
n, & \text{if } m=5,\\[3pt]
n, & \text{if } m=6 \text{ and } n \text{ is even},\\[3pt]
n+1, & \text{if } m=6 \text{ and } n \text{ is odd},\\[3pt]
n+2, & \text{if } m=7,\\[3pt]
2\left(n-\left\lfloor\dfrac{n}{3}\right\rfloor\right),
& \text{if } m=8.
\end{cases}
\]
}
\end{lemma}

\begin{proof}
{
\(m=5\): At most one vertex can be selected in each fibre, and the periodic
pattern in Figure~\ref{fig:pattern-C5Pn} selects one vertex in every fibre.
Thus
\[
\rho(C_5\Box P_n)=n.
\]
}

\smallskip
\noindent
{
\(m=6\): Let \(S\) be a packing set in \(C_6\Box P_n\), and let \(F_i\)
denote the \(i\)-th fibre. In any fibre \(F_i\), closed neighborhoods of two
selected vertices are disjoint only if the vertices are at distance \(3\) on
the cycle. Hence, \(|S\cap F_i|\le 2\) for every \(i\). Moreover, if
\(|S\cap F_i|=2\), then \(S\cap F_{i-1}=S\cap F_{i+1}=\emptyset\). To see that
these bounds are attained, choose two vertices in every second fibre, at
distance \(3\) in the \(C_6\)-direction, and leave the fibres in between empty,
shifting the chosen pair alternately by one position up and down, as
illustrated in Figure~\ref{fig:pattern-C6Pn}. This produces a packing of size
\(2\lfloor n/2\rfloor=n\) when \(n\) is even. If \(n\) is odd, the same
periodic pattern ends with one additional selected fibre, yielding
\(2\lceil n/2\rceil=n+1\) chosen vertices. Therefore
\[
\rho(C_6\Box P_n)=
\begin{cases}
n, & \text{if } n \text{ is even},\\
n+1, & \text{if } n \text{ is odd}.
\end{cases}
\]
}

{
\(m=7\): In each fibre \(F_i\) we have \(|S\cap F_i|\le 2\), since two
selected vertices in the same fibre must be at cycle distance at least \(3\)
on \(C_7\). Moreover, two consecutive fibres cannot both contain two selected
vertices. Note that choosing an empty fibre cannot increase the total size of a
packing beyond the bound \(n+2\) that we aim to attain. Hence, when proving
optimality, we may assume that no fibre is left empty unless this is forced by
the packing condition. Under this assumption, if \(|S\cap F_{i+1}|=1\), then
\(|S\cap F_{i+2}|\le 2\). Moreover, if two vertices were chosen in
\(F_{i+2}\), then fibre \(F_{i+3}\) would have to be left empty, a
contradiction.
}

\smallskip
\noindent
{
To attain the bound for \(m=7\), choose two vertices in the first fibre, for
example \((0,1)\) and \((0,5)\). Next, choose one vertex in each fibre
\(F_1,\dots,F_{n-2}\), such that each selected vertex is shifted three
positions downward along the cycle relative to the choice in the previous
fibre. If the chosen vertex in \(F_{n-2}\) lies in row \(j\), then in the last
fibre \(F_{n-1}\) select the two vertices in rows \(j+2\) and \(j-2\), with
indices taken modulo \(7\). These two vertices are at distance \(4\) on
\(C_7\), and each is at distance \(2\) from the vertex in \(F_{n-2}\).
Therefore, we obtain a packing of size \(n+2\). Since at most one fibre can
contain two selected vertices while every other fibre contains at most one, we
also have
\[
|S|\le (n-1)\cdot 1+2=n+2.
\]
Hence
\[
\rho(C_7\Box P_n)=n+2.
\]
The construction is illustrated in Figure~\ref{fig:pattern-C7Pn}.
}

\smallskip
\noindent
{
\(m=8\): In two consecutive fibres we select two vertices, at cycle distance
\(4\) in \(C_8\), and the next fibre is left empty; see
Figure~\ref{fig:pattern-C8Pn}. For optimality, note that
\(|S\cap F_i|\le 2\) for every fibre, and among any three consecutive fibres at
least one must contain at most one selected vertex, otherwise two vertices in
\(F_i\) and \(F_{i+2}\) would be at distance at most \(2\). Hence, each block
of three fibres contributes at most four vertices, which gives
\[
\rho(C_8\Box P_n)=
2\left(n-\left\lfloor\frac{n}{3}\right\rfloor\right).
\]
}
\end{proof}

\phantomsection\label{lem:packing-C5Pn}
\phantomsection\label{lem:packing-C6Pn}
\phantomsection\label{lem:packing-C7Pn}
\phantomsection\label{lem:packing-C8Pn}

\begin{figure}[ht!]
\centering

\begin{minipage}{0.48\textwidth}
\centering
\resizebox{0.95\textwidth}{!}{%
\begin{tikzpicture}[scale=0.75]

\foreach \i in {1,...,6}{
  \foreach \r in {0,1,2,3,4}{
    \draw (\i,-\r) rectangle (\i+1,-\r-1);
  }
}

\foreach \i in {0,...,5}{
  \node at (\i+1.5,0.6) {\small $F_{\i}$};
}

\foreach \r in {0,1,2,3,4}{
  \node[anchor=east] at (0.9,-\r-0.5) {\small $\r$};
}

\node at (1.5,-1.5) {\Large $\ast$};
\node at (2.5,-3.5) {\Large $\ast$};
\node at (3.5,-0.5) {\Large $\ast$};
\node at (4.5,-2.5) {\Large $\ast$};
\node at (5.5,-4.5) {\Large $\ast$};
\node at (6.5,-1.5) {\Large $\ast$};

\node at (8.2,-2.5) {\Large $\cdots$};

\end{tikzpicture}%
}
\captionof{figure}{A periodic packing pattern in $C_5\Box P_n$. The stars
indicate the vertices selected in the packing set in each fibre.}
\label{fig:pattern-C5Pn}
\end{minipage}
\hfill
\begin{minipage}{0.48\textwidth}
\centering
\resizebox{0.95\textwidth}{!}{%
\begin{tikzpicture}[scale=0.75]

\foreach \i in {1,...,6}{
  \foreach \r in {0,1,2,3,4,5}{
    \draw (\i,-\r) rectangle (\i+1,-\r-1);
  }
}

\foreach \i in {0,...,5}{
  \node at (\i+1.5,0.6) {\small $F_{\i}$};
}

\foreach \r in {0,1,2,3,4,5}{
  \node[anchor=east] at (0.9,-\r-0.5) {\small $\r$};
}

\node at (1.5,-0.5) {\Large $\ast$};
\node at (1.5,-3.5) {\Large $\ast$};

\node at (3.5,-1.5) {\Large $\ast$};
\node at (3.5,-4.5) {\Large $\ast$};

\node at (5.5,-0.5) {\Large $\ast$};
\node at (5.5,-3.5) {\Large $\ast$};

\node at (8.2,-3) {\Large $\cdots$};

\end{tikzpicture}%
}
\captionof{figure}{A periodic packing pattern in $C_6\Box P_n$. In every
second fibre we choose two vertices at distance $3$ in the $C_6$-direction,
and alternate the pair by a shift of $1$.}
\label{fig:pattern-C6Pn}
\end{minipage}

\medskip

\begin{minipage}{0.48\textwidth}
\centering
\resizebox{0.95\textwidth}{!}{%
\begin{tikzpicture}[scale=0.75]

\foreach \i in {1,...,7}{
  \foreach \r in {0,1,2,3,4,5,6}{
    \draw (\i,-\r) rectangle (\i+1,-\r-1);
  }
}

\foreach \i in {0,...,6}{
  \node at (\i+1.5,0.6) {\small $F_{\i}$};
}

\foreach \r in {0,1,2,3,4,5,6}{
  \node[anchor=east] at (0.9,-\r-0.5) {\small $\r$};
}

\node at (1.5,-1.5) {\Large $\ast$};
\node at (1.5,-5.5) {\Large $\ast$};

\node at (2.5,-3.5) {\Large $\ast$};
\node at (3.5,-6.5) {\Large $\ast$};
\node at (4.5,-2.5) {\Large $\ast$};
\node at (5.5,-5.5) {\Large $\ast$};
\node at (6.5,-1.5) {\Large $\ast$};
\node at (7.5,-4.5) {\Large $\ast$};

\node at (9.2,-3.5) {\Large $\cdots$};

\end{tikzpicture}%
}
\captionof{figure}{A periodic packing pattern in $C_7\Box P_n$ attaining
$\rho(C_7\Box P_n)=n+2$: one selected vertex in each fibre and two additional
selected vertices on the boundary fibres.}
\label{fig:pattern-C7Pn}
\end{minipage}
\hfill
\begin{minipage}{0.48\textwidth}
\centering
\resizebox{0.95\textwidth}{!}{%
\begin{tikzpicture}[scale=0.65]

\foreach \i in {1,...,9}{
  \foreach \r in {0,1,2,3,4,5,6,7}{
    \draw (\i,-\r) rectangle (\i+1,-\r-1);
  }
}

\foreach \i in {0,...,8}{
  \node at (\i+1.5,0.6) {\small $F_{\i}$};
}

\foreach \r in {0,...,7}{
  \node[anchor=east] at (0.9,-\r-0.5) {\small $\r$};
}

\node at (1.5,-0.5) {\Large $\ast$};
\node at (1.5,-4.5) {\Large $\ast$};

\node at (2.5,-2.5) {\Large $\ast$};
\node at (2.5,-6.5) {\Large $\ast$};

\node at (4.5,-0.5) {\Large $\ast$};
\node at (4.5,-4.5) {\Large $\ast$};

\node at (5.5,-2.5) {\Large $\ast$};
\node at (5.5,-6.5) {\Large $\ast$};

\node at (7.5,-0.5) {\Large $\ast$};
\node at (7.5,-4.5) {\Large $\ast$};

\node at (8.5,-2.5) {\Large $\ast$};
\node at (8.5,-6.5) {\Large $\ast$};

\node at (11,-4) {\Large $\cdots$};

\end{tikzpicture}%
}
\captionof{figure}{A periodic packing pattern in \(C_8\Box P_n\) attaining
\(\rho(C_8\Box P_n)=2(n-\lfloor n/3\rfloor)\). The pattern selects two
vertices in two consecutive fibres and leaves every third fibre empty.}
\label{fig:pattern-C8Pn}
\end{minipage}

\end{figure}

\subsection{Packing-refined upper bounds}

\smallskip
\noindent
{
{
The uniform construction above often gives closed neighborhood sums with some
surplus. To exploit this, we increase the interior base weight from
\(\left\lceil(k+4)/5\right\rceil\) to
\(\left\lceil(k+5)/5\right\rceil\), which creates one unit of surplus in every
relevant closed neighborhood. We then decrease the weight by one on every
vertex of a maximum packing set. Since the closed neighborhoods of packing
vertices are pairwise disjoint, these reductions do not interfere with one
another. Hence each reduction uses only the surplus in the corresponding closed
neighborhood and the resulting labeling still satisfies the \( [k]\)-Roman
domination condition. This gives the following packing-refined upper bound.
}
}

\begin{theorem}\label{thm:packing-refined-bounds}
{
Let \(m\in\{5,6,7,8\}\) and \(n\ge4\). Then
\begin{equation}\label{eq:packing-refined-general}
\gamma_{[kR]}(C_m\Box P_n)
\le
m (n-2) \left\lceil\frac{k+5}{5}\right\rceil
+
2m\left\lceil
\frac{k+4-\left\lceil\frac{k+5}{5}\right\rceil}{3}
\right\rceil
-\rho(C_m\Box P_n),
\end{equation}
where
\[
\rho(C_m\Box P_n)=
\begin{cases}
n, & m=5,\\[3pt]
n, & m=6 \text{ and } n \text{ is even},\\[3pt]
n+1, & m=6 \text{ and } n \text{ is odd},\\[3pt]
n+2, & m=7,\\[3pt]
2\left(n-\left\lfloor\dfrac{n}{3}\right\rfloor\right), & m=8.
\end{cases}
\]
}
\end{theorem}   

\begin{proof}
{
Let
\[
c=\left\lceil\frac{k+5}{5}\right\rceil,
\qquad
d=\left\lceil\frac{k+4-c}{3}\right\rceil .
\]
Assign weight \(c\) to every vertex in the interior fibres
\(F_1,\ldots,F_{n-2}\), and weight \(d\) to every vertex in the two boundary
fibres \(F_0\) and \(F_{n-1}\).
}

\smallskip
\noindent
{
With this choice, every closed neighborhood has at least one unit of surplus
with respect to the \( [k]\)-Roman domination condition. Indeed, if \(v\) is an
interior vertex, then
\[
f(N[v])\ge 5c\ge k+5.
\]
Since an interior vertex has at most four active neighbors, this is at least
\(k+|AN(v)|+1\). If \(v\) lies in a boundary fibre, then
\[
f(N[v])\ge c+3d\ge k+4.
\]
Since a boundary vertex has at most three active neighbors, this gives the
required surplus for the boundary vertices as well.
}

\smallskip
\noindent
{
Let \(S\) be a maximum packing set in \(C_m\Box P_n\). Decrease the weight by
one on each vertex of \(S\). Since the closed neighborhoods of vertices in
\(S\) are pairwise disjoint, each closed neighborhood contains at most one
vertex of \(S\), and therefore loses at most one unit of weight. The surplus
established above guarantees that the resulting function remains a
\([k]\)-Roman dominating function.
}

\smallskip
\noindent
{
The weight of the initial uniform labeling is
\[
m(n-2)c+2md.
\]
The packing reduction decreases this weight by
\(|S|=\rho(C_m\Box P_n)\). Hence
\[
\gamma_{[kR]}(C_m\Box P_n)
\le
m(n-2)c+2md-\rho(C_m\Box P_n),
\]
which is the claimed bound.}
\end{proof}

\phantomsection\label{C5Pn_2}
\phantomsection\label{C6Pn_2}
\phantomsection\label{thm:C7Pn_packing}
\phantomsection\label{thm:C8Pn_packing}

\subsection{Comparison of the bounds}

\smallskip
\noindent
{
For each \(m\in\{5,6,7,8\}\), we now compare the sparse periodic bound, the
uniform local-density bound, and the packing-refined bound. The comparison is
not uniform in \(k\) and \(n\); therefore we retain the comparison figures and
the asymptotic comparison results from the original presentation.
}

\subsubsection*{\texorpdfstring{The case \(m=5\)}{The case m = 5}}

\smallskip
\noindent
{
For \(C_5\Box P_n\), we compare the sparse periodic bound
\eqref{eq:ub-C5Pn-linear}, the uniform bound
\eqref{eq:uniform-upper-general}, and the packing-refined bound
\eqref{eq:packing-refined-general}. Figure~\ref{fig:best-bounds} shows which one is smallest for
\(4\le n\le20\) and \(1\le k\le22\).
}

\begin{figure}[ht!]
\centering
\begin{tikzpicture}[
x=0.50cm,
y=0.40cm,
every node/.style={font=\small}
]

\def\nmin{4}
\def\nmax{20}
\def\kmin{1}
\def\kmax{22}

\def\gapshift{5}
\def\gapsize{0.4}

\foreach \n in {\nmin,...,\nmax} {

  \foreach \k in {1,2} {

    \pgfmathtruncatemacro{\vone}{\n*(\k+1) + 2*\k}

    \pgfmathtruncatemacro{\A}{ceil((\k+4)/5)}
    \pgfmathtruncatemacro{\Btwo}{ceil((\k+3-\A)/3)}
    \pgfmathtruncatemacro{\vtwo}{5*(\n-2)*\A + 10*\Btwo}

    \pgfmathtruncatemacro{\Athree}{ceil((\k+5)/5)}
    \pgfmathtruncatemacro{\Bthree}{ceil((\k+4-\Athree)/3)}
    \pgfmathtruncatemacro{\vthree}{5*(\n-2)*\Athree + 10*\Bthree - \n}

    \def\pat{north east lines}
    \pgfmathtruncatemacro{\best}{\vone}

    \ifnum\vtwo<\best
      \def\pat{vertical lines}
      \pgfmathtruncatemacro{\best}{\vtwo}
    \fi

    \ifnum\vthree<\best
      \def\pat{crosshatch}
      \pgfmathtruncatemacro{\best}{\vthree}
    \fi

    \pgfmathsetmacro{\x}{\n-\nmin}
    \pgfmathsetmacro{\y}{\k-\kmin}

    \fill[pattern=\pat] (\x,\y) rectangle ++(1,1);
    \draw (\x,\y) rectangle ++(1,1);
  }

  \foreach \k in {8,...,22} {

    \pgfmathtruncatemacro{\vone}{\n*(\k+1) + 2*\k}

    \pgfmathtruncatemacro{\A}{ceil((\k+4)/5)}
    \pgfmathtruncatemacro{\Btwo}{ceil((\k+3-\A)/3)}
    \pgfmathtruncatemacro{\vtwo}{5*(\n-2)*\A + 10*\Btwo}

    \pgfmathtruncatemacro{\Athree}{ceil((\k+5)/5)}
    \pgfmathtruncatemacro{\Bthree}{ceil((\k+4-\Athree)/3)}
    \pgfmathtruncatemacro{\vthree}{5*(\n-2)*\Athree + 10*\Bthree - \n}

    \def\pat{north east lines}
    \pgfmathtruncatemacro{\best}{\vone}

    \ifnum\vtwo<\best
      \def\pat{vertical lines}
      \pgfmathtruncatemacro{\best}{\vtwo}
    \fi

    \ifnum\vthree<\best
      \def\pat{crosshatch}
      \pgfmathtruncatemacro{\best}{\vthree}
    \fi

    \pgfmathsetmacro{\x}{\n-\nmin}
    \pgfmathsetmacro{\y}{\k-\kmin-\gapshift+\gapsize}

    \fill[pattern=\pat] (\x,\y) rectangle ++(1,1);
    \draw (\x,\y) rectangle ++(1,1);
  }
}

\draw[->] (0,0) -- ({(\nmax-\nmin+1)+0.8},0) node[below] {$n$};

\pgfmathsetmacro{\ymax}{(\kmax-\kmin-\gapshift+\gapsize)+1.8}
\draw[->] (0,0) -- (0,\ymax) node[left] {$k$};

\foreach \n in {\nmin,...,\nmax}{
  \pgfmathsetmacro{\x}{(\n-\nmin)+0.5}
  \draw (\x,0) -- (\x,-0.18);
  \node[below] at (\x,-0.18) {\scriptsize \n};
}

\foreach \k in {1,2}{
  \pgfmathsetmacro{\yt}{(\k-\kmin)+0.5}
  \draw (0,\yt) -- (-0.18,\yt);
  \node[left] at (-0.18,\yt) {\scriptsize \k};
}

\foreach \k in {8,...,22}{
  \pgfmathsetmacro{\yt}{(\k-\kmin-\gapshift+\gapsize)+0.5}
  \draw (0,\yt) -- (-0.18,\yt);
  \node[left] at (-0.18,\yt) {\scriptsize \k};
}

\end{tikzpicture}

\vspace{2mm}

\begin{adjustbox}{max width=\linewidth}
\begin{tabular}{@{}l@{\qquad}l@{\qquad}l@{}}
\tikz[baseline=-0.5ex]{\draw[pattern=north east lines] (0,0) rectangle (0.32,0.20);}
\enspace Sparse periodic bound \eqref{eq:ub-C5Pn-linear} &
\tikz[baseline=-0.5ex]{\draw[pattern=vertical lines] (0,0) rectangle (0.32,0.20);}
\enspace Uniform bound \eqref{eq:uniform-upper-general} &
\tikz[baseline=-0.5ex]{\draw[pattern=crosshatch] (0,0) rectangle (0.32,0.20);}
\enspace Packing-refined bound \eqref{eq:packing-refined-general}
\end{tabular}
\end{adjustbox}

\caption{Best bound for \(C_5\Box P_n\) among the sparse periodic bound
\eqref{eq:ub-C5Pn-linear}, the uniform bound \eqref{eq:uniform-upper-general},
and the packing-refined bound \eqref{eq:packing-refined-general}, for
\(4\le n\le20\) and \(1\le k\le22\).}
\label{fig:best-bounds}
\end{figure}

\begin{theorem}\label{thm:bound1-eventually-optimal}
{
{Assume $m=5$.}
For every integer \(k\ge2\), there exists \(N(k)\in\mathbb{N}\) such that for
all \(n\ge N(k)\), the sparse periodic bound \eqref{eq:ub-C5Pn-linear} is
better than both the uniform bound \eqref{eq:uniform-upper-general} and the
packing-refined bound \eqref{eq:packing-refined-general}. 
}
\end{theorem}

\begin{proof}
{
For fixed \(k\), all three bounds are linear functions of \(n\) up to
constants. The slope of the sparse periodic bound \eqref{eq:ub-C5Pn-linear} is
\(k+1\). For \(m=5\), the slope of the uniform bound
\eqref{eq:uniform-upper-general} is
\[
5\left\lceil\frac{k+4}{5}\right\rceil>k+1.
\]
The slope of the packing-refined bound \eqref{eq:packing-refined-general} is
\[
5\left\lceil\frac{k+5}{5}\right\rceil-1>k+1,
\]
because \(\rho(C_5\Box P_n)=n\). Therefore, for every fixed \(k\ge2\), the
sparse periodic bound is eventually smaller than both the uniform and the
packing-refined bounds.
}
\end{proof}

\subsubsection*{\texorpdfstring{The case \(m=6\)}{The case m = 6}}

\smallskip
\noindent
{
For \(C_6\Box P_n\), Figures~\ref{fig:best-bounds-C6Pn-k1-2-14-25-n4-20}
and \ref{fig:best-bounds-C6Pn-k40-57-n90-110} compare the sparse periodic
bound \eqref{eq:ub-C6Pn-linear} with the uniform bound
\eqref{eq:uniform-upper-general} and the packing-refined bound
\eqref{eq:packing-refined-general}.
}

\begin{figure}[ht!]
\centering
\begin{tikzpicture}[
  x=0.50cm,
  y=0.40cm,
  every node/.style={font=\small}
]
\def\nmin{4}
\def\nmax{20}
\def\kmin{1}
\def\kmax{25}

\def\gapshift{11}
\def\gapsize{0.4}

\foreach \n in {\nmin,...,\nmax} {

  \foreach \k in {1,2} {

    \pgfmathtruncatemacro{\c}{ceil(4*\n/3)}
    \pgfmathtruncatemacro{\r}{mod(\n,3)}
    \pgfmathtruncatemacro{\vonebase}{\c*(\k+1)}
    \pgfmathtruncatemacro{\corr}{0}
    \ifnum\r=0 \pgfmathtruncatemacro{\corr}{\k+1}\fi
    \ifnum\r=2 \pgfmathtruncatemacro{\corr}{\k}\fi
    \pgfmathtruncatemacro{\vone}{\vonebase+\corr}

    \pgfmathtruncatemacro{\A}{ceil((\k+4)/5)}
    \pgfmathtruncatemacro{\B}{ceil((\k+3-\A)/3)}
    \pgfmathtruncatemacro{\vtwo}{6*(\n-2)*\A + 12*\B}

    \pgfmathtruncatemacro{\Athree}{ceil((\k+5)/5)}
    \pgfmathtruncatemacro{\Bthree}{ceil((\k+4-\Athree)/3)}
    \pgfmathtruncatemacro{\cn}{ceil(\n/2)}
    \pgfmathtruncatemacro{\vthree}{6*(\n-2)*\Athree + 12*\Bthree - 2*\cn}

    \def\pat{north east lines}
    \pgfmathtruncatemacro{\best}{\vone}

    \ifnum\vtwo<\best
      \def\pat{vertical lines}
      \pgfmathtruncatemacro{\best}{\vtwo}
    \fi

    \ifnum\vthree<\best
      \def\pat{crosshatch}
      \pgfmathtruncatemacro{\best}{\vthree}
    \fi

    \pgfmathsetmacro{\x}{\n-\nmin}
    \pgfmathsetmacro{\y}{\k-\kmin}

    \fill[pattern=\pat] (\x,\y) rectangle ++(1,1);
    \draw[line width=0.2pt] (\x,\y) rectangle ++(1,1);
  }

  \foreach \k in {14,...,25} {

    \pgfmathtruncatemacro{\c}{ceil(4*\n/3)}
    \pgfmathtruncatemacro{\r}{mod(\n,3)}
    \pgfmathtruncatemacro{\vonebase}{\c*(\k+1)}
    \pgfmathtruncatemacro{\corr}{0}
    \ifnum\r=0 \pgfmathtruncatemacro{\corr}{\k+1}\fi
    \ifnum\r=2 \pgfmathtruncatemacro{\corr}{\k}\fi
    \pgfmathtruncatemacro{\vone}{\vonebase+\corr}

    \pgfmathtruncatemacro{\A}{ceil((\k+4)/5)}
    \pgfmathtruncatemacro{\B}{ceil((\k+3-\A)/3)}
    \pgfmathtruncatemacro{\vtwo}{6*(\n-2)*\A + 12*\B}

    \pgfmathtruncatemacro{\Athree}{ceil((\k+5)/5)}
    \pgfmathtruncatemacro{\Bthree}{ceil((\k+4-\Athree)/3)}
    \pgfmathtruncatemacro{\cn}{ceil(\n/2)}
    \pgfmathtruncatemacro{\vthree}{6*(\n-2)*\Athree + 12*\Bthree - 2*\cn}

    \def\pat{north east lines}
    \pgfmathtruncatemacro{\best}{\vone}

    \ifnum\vtwo<\best
      \def\pat{vertical lines}
      \pgfmathtruncatemacro{\best}{\vtwo}
    \fi

    \ifnum\vthree<\best
      \def\pat{crosshatch}
      \pgfmathtruncatemacro{\best}{\vthree}
    \fi

    \pgfmathsetmacro{\x}{\n-\nmin}
    \pgfmathsetmacro{\y}{\k-\kmin-\gapshift+\gapsize}

    \fill[pattern=\pat] (\x,\y) rectangle ++(1,1);
    \draw[line width=0.2pt] (\x,\y) rectangle ++(1,1);
  }
}

\draw[->] (0,0) -- ({(\nmax-\nmin+1)+0.8},0) node[below] {$n$};

\pgfmathsetmacro{\ymax}{(\kmax-\kmin-\gapshift+\gapsize)+1.8}
\draw[->] (0,0) -- (0,\ymax) node[left] {$k$};

\foreach \n in {\nmin,...,\nmax} {
  \pgfmathsetmacro{\x}{(\n-\nmin)+0.5}
  \draw (\x,0) -- (\x,-0.18);
  \node[below] at (\x,-0.18) {\scriptsize \n};
}

\foreach \k in {1,2} {
  \pgfmathsetmacro{\yt}{(\k-\kmin)+0.5}
  \draw (0,\yt) -- (-0.18,\yt);
  \node[left] at (-0.18,\yt) {\scriptsize \k};
}

\foreach \k in {14,...,25} {
  \pgfmathsetmacro{\yt}{(\k-\kmin-\gapshift+\gapsize)+0.5}
  \draw (0,\yt) -- (-0.18,\yt);
  \node[left] at (-0.18,\yt) {\scriptsize \k};
}

\end{tikzpicture}

\vspace{2mm}

\begin{adjustbox}{max width=\linewidth}
\begin{tabular}{@{}l@{\qquad}l@{\qquad}l@{}}
\tikz[baseline=-0.5ex]{\draw[pattern=north east lines] (0,0) rectangle (0.32,0.20);}
\enspace Sparse periodic bound \eqref{eq:ub-C6Pn-linear} &
\tikz[baseline=-0.5ex]{\draw[pattern=vertical lines] (0,0) rectangle (0.32,0.20);}
\enspace Uniform bound \eqref{eq:uniform-upper-general} &
\tikz[baseline=-0.5ex]{\draw[pattern=crosshatch] (0,0) rectangle (0.32,0.20);}
\enspace Packing-refined bound \eqref{eq:packing-refined-general}
\end{tabular}
\end{adjustbox}

\caption{Best bound for \(C_6\Box P_n\) among the sparse periodic bound
\eqref{eq:ub-C6Pn-linear}, the uniform bound \eqref{eq:uniform-upper-general},
and the packing-refined bound \eqref{eq:packing-refined-general}, for
\(4\le n\le20\) and \(k=1,2,14,\ldots,25\).}
\label{fig:best-bounds-C6Pn-k1-2-14-25-n4-20}
\end{figure}

\begin{figure}[ht!]
\centering
\begin{tikzpicture}[
  x=0.50cm,
  y=0.40cm,
  every node/.style={font=\small}
]
\def\nmin{90}
\def\nmax{110}
\def\kmin{40}
\def\kmax{57}

\foreach \n in {\nmin,...,\nmax} {
  \foreach \k in {\kmin,...,\kmax} {

    \pgfmathtruncatemacro{\c}{ceil(4*\n/3)}
    \pgfmathtruncatemacro{\r}{mod(\n,3)}
    \pgfmathtruncatemacro{\vonebase}{\c*(\k+1)}
    \pgfmathtruncatemacro{\corr}{0}
    \ifnum\r=0 \pgfmathtruncatemacro{\corr}{\k+1}\fi
    \ifnum\r=2 \pgfmathtruncatemacro{\corr}{\k}\fi
    \pgfmathtruncatemacro{\vone}{\vonebase+\corr}

    \pgfmathtruncatemacro{\A}{ceil((\k+4)/5)}
    \pgfmathtruncatemacro{\B}{ceil((\k+3-\A)/3)}
    \pgfmathtruncatemacro{\vtwo}{6*(\n-2)*\A + 12*\B}

    \pgfmathtruncatemacro{\Athree}{ceil((\k+5)/5)}
    \pgfmathtruncatemacro{\Bthree}{ceil((\k+4-\Athree)/3)}
    \pgfmathtruncatemacro{\cn}{ceil(\n/2)}
    \pgfmathtruncatemacro{\vthree}{6*(\n-2)*\Athree + 12*\Bthree - 2*\cn}

    \def\pat{north east lines}
    \pgfmathtruncatemacro{\best}{\vone}

    \ifnum\vtwo<\best
      \def\pat{vertical lines}
      \pgfmathtruncatemacro{\best}{\vtwo}
    \fi

    \ifnum\vthree<\best
      \def\pat{crosshatch}
      \pgfmathtruncatemacro{\best}{\vthree}
    \fi

    \pgfmathsetmacro{\x}{\n-\nmin}
    \pgfmathsetmacro{\y}{\k-\kmin}

    \fill[pattern=\pat] (\x,\y) rectangle ++(1,1);
    \draw[line width=0.2pt] (\x,\y) rectangle ++(1,1);
  }
}

\draw[->] (0,0) -- ({(\nmax-\nmin+1)+0.8},0) node[below] {$n$};
\draw[->] (0,0) -- (0,{(\kmax-\kmin+1)+0.8}) node[left] {$k$};

\foreach \n in {\nmin,...,\nmax} {
  \pgfmathsetmacro{\x}{(\n-\nmin)+0.5}
  \draw (\x,0) -- (\x,-0.18);
  \pgfmathtruncatemacro{\labeltick}{mod(\n,5)}
  \ifnum\labeltick=0
    \node[below] at (\x,-0.18) {\scriptsize \n};
  \fi
}

\foreach \k in {\kmin,...,\kmax} {
  \pgfmathsetmacro{\yt}{(\k-\kmin)+0.5}
  \draw (0,\yt) -- (-0.18,\yt);
  \node[left] at (-0.18,\yt) {\scriptsize \k};
}

\end{tikzpicture}

\vspace{2mm}

\begin{adjustbox}{max width=\linewidth}
\begin{tabular}{@{}l@{\qquad}l@{\qquad}l@{}}
\tikz[baseline=-0.5ex]{\draw[pattern=north east lines] (0,0) rectangle (0.32,0.20);}
\enspace Sparse periodic bound \eqref{eq:ub-C6Pn-linear} &
\tikz[baseline=-0.5ex]{\draw[pattern=vertical lines] (0,0) rectangle (0.32,0.20);}
\enspace Uniform bound \eqref{eq:uniform-upper-general} &
\tikz[baseline=-0.5ex]{\draw[pattern=crosshatch] (0,0) rectangle (0.32,0.20);}
\enspace Packing-refined bound \eqref{eq:packing-refined-general}
\end{tabular}
\end{adjustbox}

\caption{Best bound for \(C_6\Box P_n\) among the sparse periodic bound
\eqref{eq:ub-C6Pn-linear}, the uniform bound \eqref{eq:uniform-upper-general},
and the packing-refined bound \eqref{eq:packing-refined-general}, for
\(90\le n\le110\) and \(40\le k\le57\).}
\label{fig:best-bounds-C6Pn-k40-57-n90-110}
\end{figure}

\smallskip
\noindent
{
Figure \ref{fig:best-bounds-C6Pn-k40-57-n90-110} shows the transition point. For \(k=52\), the sparse periodic bound
\eqref{eq:ub-C6Pn-linear} is still the smallest for \(n\ge102\), while at
\(k=53\) it is overtaken by the packing-refined bound
\eqref{eq:packing-refined-general}. Moreover, for every \(k\equiv1\pmod5\),
the uniform bound \eqref{eq:uniform-upper-general} becomes optimal for large
\(k\).
}

\begin{theorem}\label{thm:C6Pn-asymptotic-threshold}
{{Assume  \(m=6\).}
For all sufficiently large \(n\), at least one of the uniform bound
\eqref{eq:uniform-upper-general} or the packing-refined bound
\eqref{eq:packing-refined-general} improves the sparse
periodic bound \eqref{eq:ub-C6Pn-linear}. More precisely:
\begin{itemize}
\item if \(k\equiv1\pmod5\), then for \(k\ge31\) and all sufficiently large
      \(n\), the uniform bound is the smallest among the three bounds;
\item if \(k\not\equiv1\pmod5\), then for \(k\ge53\) and all sufficiently large
      \(n\), the packing-refined bound is the smallest among the three bounds.
\end{itemize}
}
\end{theorem}

\begin{proof}
{
Fix \(k\) and compare the three bounds as functions of \(n\). Since
\[
\left\lceil\frac{4n}{3}\right\rceil
=
\frac43 n+O(1),
\]
the sparse periodic bound \eqref{eq:ub-C6Pn-linear} has the form
\[
\frac43(k+1)n+O_k(1).
\]
}

\smallskip
\noindent
{
Let
\[
a=\left\lceil\frac{k+4}{5}\right\rceil
\qquad\text{and}\qquad
b=\left\lceil\frac{k+3-a}{3}\right\rceil .
\]
Then the uniform bound \eqref{eq:uniform-upper-general}, with \(m=6\), equals
\[
6(n-2)a+12b=6an+O_k(1).
\]
}

\smallskip
\noindent
{
Let
\[
a'=\left\lceil\frac{k+5}{5}\right\rceil
\qquad\text{and}\qquad
b'=\left\lceil\frac{k+4-a'}{3}\right\rceil .
\]
Since
\[
\rho(C_6\Box P_n)=n+O(1),
\]
the packing-refined bound \eqref{eq:packing-refined-general}, with \(m=6\),
equals
\[
6(n-2)a'+12b'-\rho(C_6\Box P_n)
=
(6a'-1)n+O_k(1).
\]
}

\smallskip
\noindent
{
Therefore, for sufficiently large \(n\), the smallest bound is determined by
comparing the linear coefficients
\[
\frac43(k+1),
\qquad
6a,
\qquad
6a'-1.
\]
}

\smallskip
\noindent
{
Comparing \(6a\) with \(\frac43(k+1)\) shows that
\[
6a<\frac43(k+1)
\]
holds precisely in the residue class \(k\equiv1\pmod5\), once \(k\ge31\).
In the remaining residue classes,
\[
6a'-1<\frac43(k+1)
\]
holds once \(k\ge53\). Finally, if \(k\equiv1\pmod5\), then
\(a<a'\), and hence \(6a<6a'-1\), so the uniform bound is asymptotically
smaller than the packing-refined bound. If \(k\not\equiv1\pmod5\), then
\(a=a'\), and hence \(6a'-1<6a\), so the packing-refined bound is
asymptotically smaller than the uniform bound. This completes the proof.
}
\end{proof}

\subsubsection*{\texorpdfstring{The case \(m=7\)}{The case m = 7}}

\smallskip
\noindent
{
For \(C_7\Box P_n\), Figure~\ref{fig:best-bounds-C7Pn-k1-2-12-20-n4-20}
compares the sparse periodic bound \eqref{eq:ub-C7Pn-linear} with the
uniform bound \eqref{eq:uniform-upper-general} and the packing-refined bound
\eqref{eq:packing-refined-general}.
}

\begin{figure}[ht!]
\centering
\begin{tikzpicture}[
x=0.50cm,
y=0.33cm,
every node/.style={font=\small}
]

\def\nmin{4}
\def\nmax{20}
\def\kmin{1}
\def\kmax{20}

\def\gapshift{9}
\def\gapsize{0.4}

\def\yy#1{%
  \ifnum#1<3
    \pgfmathsetmacro{\y}{#1-\kmin}
  \else
    \pgfmathsetmacro{\y}{#1-\kmin-\gapshift+\gapsize}
  \fi
}

\foreach \n in {\nmin,...,\nmax} {

  \foreach \k in {1,2,12,13,14,15,16,17,18,19,20} {

    \ifodd\n
      \pgfmathtruncatemacro{\vone}{((3*\n+1)/2)*(\k+1) + 2*\k}
    \else
      \pgfmathtruncatemacro{\vone}{(3*\n/2)*(\k+1) + 2*\k + 1}
    \fi

    \pgfmathtruncatemacro{\A}{ceil((\k+4)/5)}
    \pgfmathtruncatemacro{\B}{ceil((\k+3-\A)/3)}
    \pgfmathtruncatemacro{\vtwo}{7*(\n-2)*\A + 14*\B}

    \pgfmathtruncatemacro{\Athree}{ceil((\k+5)/5)}
    \pgfmathtruncatemacro{\Bthree}{ceil((\k+4-\Athree)/3)}
    \pgfmathtruncatemacro{\vthree}{7*(\n-2)*\Athree + 14*\Bthree - (\n+2)}

    \def\pat{north east lines}
    \pgfmathtruncatemacro{\best}{\vone}

    \ifnum\vtwo<\best
      \def\pat{vertical lines}
      \pgfmathtruncatemacro{\best}{\vtwo}
    \fi

    \ifnum\vthree<\best
      \def\pat{crosshatch}
      \pgfmathtruncatemacro{\best}{\vthree}
    \fi

    \pgfmathsetmacro{\x}{\n-\nmin}
    \yy{\k}

    \fill[pattern=\pat] (\x,\y) rectangle ++(1,1);
    \draw[line width=0.2pt] (\x,\y) rectangle ++(1,1);
  }
}

\draw[->] (0,0) -- ({(\nmax-\nmin+1)+0.8},0) node[below] {$n$};

\pgfmathsetmacro{\ymax}{(\kmax-\kmin-\gapshift+\gapsize)+1.8}
\draw[->] (0,0) -- (0,\ymax) node[left] {$k$};

\foreach \n in {\nmin,...,\nmax} {
  \pgfmathsetmacro{\x}{(\n-\nmin)+0.5}
  \draw (\x,0) -- (\x,-0.18);
  \node[below] at (\x,-0.18) {\scriptsize \n};
}

\foreach \k in {1,2} {
  \pgfmathsetmacro{\yt}{(\k-\kmin)+0.5}
  \draw (0,\yt) -- (-0.18,\yt);
  \node[left] at (-0.18,\yt) {\scriptsize \k};
}

\foreach \k in {12,13,14,15,16,17,18,19,20} {
  \pgfmathsetmacro{\yt}{(\k-\kmin-\gapshift+\gapsize)+0.5}
  \draw (0,\yt) -- (-0.18,\yt);
  \node[left] at (-0.18,\yt) {\scriptsize \k};
}

\end{tikzpicture}

\vspace{2mm}

\begin{adjustbox}{max width=\linewidth}
\begin{tabular}{@{}l@{\qquad}l@{\qquad}l@{}}
\tikz[baseline=-0.5ex]{\draw[pattern=north east lines] (0,0) rectangle (0.32,0.20);}
\enspace Sparse periodic bound \eqref{eq:ub-C7Pn-linear} &
\tikz[baseline=-0.5ex]{\draw[pattern=vertical lines] (0,0) rectangle (0.32,0.20);}
\enspace Uniform bound \eqref{eq:uniform-upper-general} &
\tikz[baseline=-0.5ex]{\draw[pattern=crosshatch] (0,0) rectangle (0.32,0.20);}
\enspace Packing-refined bound \eqref{eq:packing-refined-general}
\end{tabular}
\end{adjustbox}

\caption{Best bound for \(C_7\Box P_n\) among the sparse periodic bound
\eqref{eq:ub-C7Pn-linear}, the uniform bound \eqref{eq:uniform-upper-general},
and the packing-refined bound \eqref{eq:packing-refined-general}, for
\(4\le n\le20\) and \(k=1,2,12,\ldots,20\).}
\label{fig:best-bounds-C7Pn-k1-2-12-20-n4-20}
\end{figure}

\smallskip
\noindent
{
In this parameter range, the sparse periodic bound
\eqref{eq:ub-C7Pn-linear} dominates in the majority of cases. However, a
visible transition region appears for intermediate values of \(k\). In
particular, the uniform bound \eqref{eq:uniform-upper-general} becomes optimal
for some smaller values of \(n\), while the packing-refined bound
\eqref{eq:packing-refined-general} begins to dominate for larger values of
\(k\).
}

\smallskip
\noindent
{
We now determine a uniform threshold beyond which one of the density-type
bounds asymptotically improves the sparse periodic bound. Since all three expressions are linear in $n$, up to additive terms depending only on $k$, their eventual comparison is governed by the leading coefficient with respect to $n$.
}

\begin{theorem}\label{thm:C7Pn-asymptotic-threshold}
{{Assume  \(m=7\).}
For all sufficiently large \(n\), at least one of the uniform bound
\eqref{eq:uniform-upper-general} or the packing-refined bound
\eqref{eq:packing-refined-general}  improves the sparse
periodic bound \eqref{eq:ub-C7Pn-linear}. More precisely:
\begin{itemize}
\item if \(k\equiv1\pmod5\), then for \(k\ge46\) and all sufficiently large
      \(n\), the uniform bound is the smallest among the three bounds;
\item if \(k\not\equiv1\pmod5\), then for \(k\ge87\) and all sufficiently large
      \(n\), the packing-refined bound is the smallest among the three bounds.
\end{itemize}
}
\end{theorem}

\begin{proof}
{
Fix \(k\) and compare the three bounds as functions of \(n\). The sparse
periodic bound \eqref{eq:ub-C7Pn-linear} has the form
\[
\frac32(k+1)n+O_k(1).
\]
}

\smallskip
\noindent
{
Let
\[
a=\left\lceil\frac{k+4}{5}\right\rceil
\qquad\text{and}\qquad
b=\left\lceil\frac{k+3-a}{3}\right\rceil .
\]
Then the uniform bound \eqref{eq:uniform-upper-general}, with \(m=7\), equals
\[
7(n-2)a+14b=7an+O_k(1).
\]
}

\smallskip
\noindent
{
Let
\[
a'=\left\lceil\frac{k+5}{5}\right\rceil
\qquad\text{and}\qquad
b'=\left\lceil\frac{k+4-a'}{3}\right\rceil .
\]
Since
\[
\rho(C_7\Box P_n)=n+2=n+O(1),
\]
the packing-refined bound \eqref{eq:packing-refined-general}, with \(m=7\),
equals
\[
7(n-2)a'+14b'-\rho(C_7\Box P_n)
=
(7a'-1)n+O_k(1).
\]
}

\smallskip
\noindent
{
Therefore, for sufficiently large \(n\), the smallest bound is determined by
comparing the linear coefficients
\[
\frac32(k+1),
\qquad
7a,
\qquad
7a'-1.
\]
}

\smallskip
\noindent
{
Comparing \(7a\) with \(\frac32(k+1)\) shows that
\[
7a<\frac32(k+1)
\]
holds in the residue class \(k\equiv1\pmod5\) once \(k\ge46\). It remains to compare it with the sparse periodic bound. The comparison
\[
7a'-1<\frac32(k+1)
\]
holds for all \(k\not\equiv1\pmod5\) with \(k\ge88\), except that the first
relevant value \(k=87\) gives equality of the two slopes. For \(k=87\), we have
\[
\frac32(k+1)=132
\qquad\text{and}\qquad
7\left\lceil\frac{k+5}{5}\right\rceil-1=132.
\]
Thus the sparse periodic and packing-refined bounds have the same asymptotic
slope. However, their constant terms differ. Indeed, the packing-refined bound
equals
\[
7(n-2)\cdot 19+14\cdot 24-(n+2)=132n+68.
\]
On the other hand, the sparse periodic bound equals
$132n+175$
when \(n\) is even, and
$132n+218$
when \(n\) is odd. Hence, for \(k=87\), the packing-refined bound is smaller
than the sparse periodic bound for all \(n\).

\smallskip
\noindent
Finally, if \(k\equiv1\pmod5\), then \(a<a'\), and
hence \(7a<7a'-1\), so the uniform bound is asymptotically smaller than the
packing-refined bound. If \(k\not\equiv1\pmod5\), then \(a=a'\), and hence
\(7a'-1<7a\), so the packing-refined bound is asymptotically smaller than the
uniform bound. This completes the proof.
}
\end{proof}

\smallskip
\noindent
{
To illustrate the transition region, we compare the three bounds for
\(260\le n\le280\) and \(75\le k\le87\).
}

\begin{figure}[ht!]
\centering
\begin{tikzpicture}[
x=0.40cm,
y=0.40cm,
every node/.style={font=\small}
]

\def\nmin{260}
\def\nmax{280}
\def\kmin{75}
\def\kmax{87}

\foreach \nn in {\nmin,...,\nmax} {
  \foreach \kk in {\kmin,...,\kmax} {

    \pgfmathtruncatemacro{\A}{ceil((\kk+4)/5)}
    \pgfmathtruncatemacro{\B}{ceil((\kk+3-\A)/3)}

    \ifodd\nn
      \pgfmathtruncatemacro{\duniform}{
        \nn*(14*\A - 3*(\kk+1)) - 28*\A + 28*\B - 5*\kk - 1
      }
    \else
      \pgfmathtruncatemacro{\duniform}{
        \nn*(14*\A - 3*(\kk+1)) - 28*\A + 28*\B - 4*\kk - 2
      }
    \fi

    \pgfmathtruncatemacro{\Athree}{ceil((\kk+5)/5)}
    \pgfmathtruncatemacro{\Bthree}{ceil((\kk+4-\Athree)/3)}

    \ifodd\nn
      \pgfmathtruncatemacro{\dpacking}{
        \nn*(14*\Athree - 3*(\kk+1) - 2)
        - 28*\Athree + 28*\Bthree - 5*\kk - 5
      }
    \else
      \pgfmathtruncatemacro{\dpacking}{
        \nn*(14*\Athree - 3*(\kk+1) - 2)
        - 28*\Athree + 28*\Bthree - 4*\kk - 6
      }
    \fi

    \def\pat{north east lines}
    \pgfmathtruncatemacro{\best}{0}

    \ifnum\duniform<\best
      \def\pat{vertical lines}
      \pgfmathtruncatemacro{\best}{\duniform}
    \fi

    \ifnum\dpacking<\best
      \def\pat{crosshatch}
      \pgfmathtruncatemacro{\best}{\dpacking}
    \fi

    \pgfmathsetmacro{\x}{\nn-\nmin}
    \pgfmathsetmacro{\y}{\kk-\kmin}

    \fill[pattern=\pat] (\x,\y) rectangle ++(1,1);
    \draw[line width=0.2pt] (\x,\y) rectangle ++(1,1);
  }
}

\draw[->] (0,0) -- ({(\nmax-\nmin+1)+0.8},0) node[below] {$n$};
\draw[->] (0,0) -- (0,{(\kmax-\kmin+1)+0.8}) node[left] {$k$};

\foreach \nn in {260,265,270,275,280} {
  \pgfmathsetmacro{\x}{(\nn-\nmin)+0.5}
  \draw (\x,0) -- (\x,-0.18);
  \node[below] at (\x,-0.18) {\scriptsize \nn};
}

\foreach \kk in {\kmin,...,\kmax} {
  \pgfmathsetmacro{\y}{(\kk-\kmin)+0.5}
  \draw (0,\y) -- (-0.18,\y);
  \node[left] at (-0.18,\y) {\scriptsize \kk};
}

\end{tikzpicture}

\vspace{2mm}

\begin{adjustbox}{max width=\linewidth}
\begin{tabular}{@{}l@{\qquad}l@{\qquad}l@{}}
\tikz[baseline=-0.5ex]{\draw[pattern=north east lines] (0,0) rectangle (0.32,0.20);}
\enspace Sparse periodic bound \eqref{eq:ub-C7Pn-linear} &
\tikz[baseline=-0.5ex]{\draw[pattern=vertical lines] (0,0) rectangle (0.32,0.20);}
\enspace Uniform bound \eqref{eq:uniform-upper-general} &
\tikz[baseline=-0.5ex]{\draw[pattern=crosshatch] (0,0) rectangle (0.32,0.20);}
\enspace Packing-refined bound \eqref{eq:packing-refined-general}
\end{tabular}
\end{adjustbox}

\caption{Best bound for \(C_7\Box P_n\) among the sparse periodic bound
\eqref{eq:ub-C7Pn-linear}, the uniform bound \eqref{eq:uniform-upper-general},
and the packing-refined bound \eqref{eq:packing-refined-general}, for
\(260\le n\le280\) and \(75\le k\le87\).}
\label{fig:best-bounds-C7Pn-k75-87-n260-280}
\end{figure}

\smallskip
\noindent
{
Figure \ref{fig:best-bounds-C7Pn-k75-87-n260-280} illustrates the transition predicted by the slope comparison. The
case \(k=87\) is the borderline case: the sparse periodic and packing-refined
bounds have the same slope, but the constant term favors the packing-refined
bound.}

\subsubsection*{\texorpdfstring{The case \(m=8\)}{The case m = 8}}

\smallskip
\noindent
{
For \(C_8\Box P_n\), Figure~\ref{fig:best-bounds-C8Pn-k1-2-8-25-n4-20}
compares the sparse periodic bound \eqref{eq:ub-C8Pn-linear} with the
uniform bound \eqref{eq:uniform-upper-general} and the packing-refined bound
\eqref{eq:packing-refined-general}.
}

\begin{figure}[ht!]
\centering
\begin{tikzpicture}[
  x=0.50cm,
  y=0.33cm,
  every node/.style={font=\small}
]

\def\nmin{15}
\def\nmax{35}
\def\kmin{1}
\def\kmax{28}

\def\gapshift{5}
\def\gapsize{0.4}

\foreach \nn in {\nmin,...,\nmax} {

  \foreach \kk in {1,2} {

    \pgfmathtruncatemacro{\vone}{2*\nn*(\kk+1)
      - (floor((\nn-2)/5) + floor(\nn/5)) + 2*\kk}

    \pgfmathtruncatemacro{\A}{ceil((\kk+4)/5)}
    \pgfmathtruncatemacro{\B}{ceil((\kk+3-\A)/3)}
    \pgfmathtruncatemacro{\vtwo}{8*(\nn-2)*\A + 16*\B}

    \pgfmathtruncatemacro{\Athree}{ceil((\kk+5)/5)}
    \pgfmathtruncatemacro{\Bthree}{ceil((\kk+4-\Athree)/3)}
    \pgfmathtruncatemacro{\F}{floor(\nn/3)}
    \pgfmathtruncatemacro{\vthree}{8*(\nn-2)*\Athree + 16*\Bthree - 2*(\nn-\F)}

    \def\pat{north east lines}
    \pgfmathtruncatemacro{\best}{\vone}

    \ifnum\vtwo<\best
      \def\pat{vertical lines}
      \pgfmathtruncatemacro{\best}{\vtwo}
    \fi
    \ifnum\vthree<\best
      \def\pat{crosshatch}
      \pgfmathtruncatemacro{\best}{\vthree}
    \fi

    \pgfmathsetmacro{\x}{\nn-\nmin}
    \pgfmathsetmacro{\y}{\kk-\kmin}

    \fill[pattern=\pat] (\x,\y) rectangle ++(1,1);
    \draw[line width=0.2pt] (\x,\y) rectangle ++(1,1);
  }

  \foreach \kk in {8,...,28} {

    \pgfmathtruncatemacro{\vone}{2*\nn*(\kk+1)
      - (floor((\nn-2)/5) + floor(\nn/5)) + 2*\kk}

    \pgfmathtruncatemacro{\A}{ceil((\kk+4)/5)}
    \pgfmathtruncatemacro{\B}{ceil((\kk+3-\A)/3)}
    \pgfmathtruncatemacro{\vtwo}{8*(\nn-2)*\A + 16*\B}

    \pgfmathtruncatemacro{\Athree}{ceil((\kk+5)/5)}
    \pgfmathtruncatemacro{\Bthree}{ceil((\kk+4-\Athree)/3)}
    \pgfmathtruncatemacro{\F}{floor(\nn/3)}
    \pgfmathtruncatemacro{\vthree}{8*(\nn-2)*\Athree + 16*\Bthree - 2*(\nn-\F)}

    \def\pat{north east lines}
    \pgfmathtruncatemacro{\best}{\vone}

    \ifnum\vtwo<\best
      \def\pat{vertical lines}
      \pgfmathtruncatemacro{\best}{\vtwo}
    \fi
    \ifnum\vthree<\best
      \def\pat{crosshatch}
      \pgfmathtruncatemacro{\best}{\vthree}
    \fi

    \pgfmathsetmacro{\x}{\nn-\nmin}
    \pgfmathsetmacro{\y}{\kk-\kmin-\gapshift+\gapsize}

    \fill[pattern=\pat] (\x,\y) rectangle ++(1,1);
    \draw[line width=0.2pt] (\x,\y) rectangle ++(1,1);
  }
}

\draw[->] (0,0) -- ({(\nmax-\nmin+1)+0.8},0) node[below] {$n$};

\pgfmathsetmacro{\ymax}{(\kmax-\kmin-\gapshift+\gapsize)+1.8}
\draw[->] (0,0) -- (0,\ymax) node[left] {$k$};

\foreach \nn in {\nmin,...,\nmax} {
  \pgfmathsetmacro{\xt}{(\nn-\nmin)+0.5}
  \draw (\xt,0) -- (\xt,-0.18);
  \node[below] at (\xt,-0.18) {\scriptsize \nn};
}

\foreach \kk in {1,2} {
  \pgfmathsetmacro{\yt}{(\kk-\kmin)+0.5}
  \draw (0,\yt) -- (-0.18,\yt);
  \node[left] at (-0.18,\yt) {\scriptsize \kk};
}

\foreach \kk in {8,...,28} {
  \pgfmathsetmacro{\yt}{(\kk-\kmin-\gapshift+\gapsize)+0.5}
  \draw (0,\yt) -- (-0.18,\yt);
  \node[left] at (-0.18,\yt) {\scriptsize \kk};
}

\end{tikzpicture}

\vspace{2mm}

\begin{adjustbox}{max width=\linewidth}
\begin{tabular}{@{}l@{\qquad}l@{\qquad}l@{}}
\tikz[baseline=-0.5ex]{\draw[pattern=north east lines] (0,0) rectangle (0.32,0.20);}
\enspace Sparse periodic bound \eqref{eq:ub-C8Pn-linear} &
\tikz[baseline=-0.5ex]{\draw[pattern=vertical lines] (0,0) rectangle (0.32,0.20);}
\enspace Uniform bound \eqref{eq:uniform-upper-general} &
\tikz[baseline=-0.5ex]{\draw[pattern=crosshatch] (0,0) rectangle (0.32,0.20);}
\enspace Packing-refined bound \eqref{eq:packing-refined-general}
\end{tabular}
\end{adjustbox}

\caption{Best bound for \(C_8\Box P_n\) among the sparse periodic bound
\eqref{eq:ub-C8Pn-linear}, the uniform bound \eqref{eq:uniform-upper-general},
and the packing-refined bound \eqref{eq:packing-refined-general}, for
\(4\le n\le20\) and \(k=1,2,8,\ldots,28\).}
\label{fig:best-bounds-C8Pn-k1-2-8-25-n4-20}
\end{figure}
\smallskip
\noindent
{
As in the previous cases, the sparse periodic bound is strongest for smaller
parameter values, while the uniform and packing-refined bounds become relevant
as \(k\) increases. We now compare the corresponding asymptotic slopes.
}

\begin{theorem}\label{thm:C8Pn-asymptotic-threshold}
{{Assume  \(m=8\).}
For all sufficiently large \(n\), at least one of the uniform bound
\eqref{eq:uniform-upper-general} or the packing-refined bound
\eqref{eq:packing-refined-general} improves the sparse
periodic bound \eqref{eq:ub-C8Pn-linear}. More precisely:
\begin{itemize}
\item if \(k\equiv1\pmod5\), then for \(k\ge16\) and all sufficiently large
      \(n\), the uniform bound is the smallest among the three bounds;
\item if \(k\not\equiv1\pmod5\), then for \(k\ge23\) and all sufficiently large
      \(n\), the packing-refined bound is the smallest among the three bounds.
\end{itemize}
}
\end{theorem}

\begin{proof}
{
For fixed \(k\), compare the asymptotic slopes in \(n\). The sparse periodic,
uniform, and packing-refined bounds have slopes
\[
2(k+1)-\frac25,\qquad
8\left\lceil\frac{k+4}{5}\right\rceil,\qquad
8\left\lceil\frac{k+5}{5}\right\rceil-\frac43,
\]
respectively. If \(k\equiv1\pmod5\) and \(k\ge16\), the second of these is
strictly smaller than the other two, so the uniform bound is eventually the
smallest. If \(k\not\equiv1\pmod5\) and \(k\ge23\), then
\[
8\left\lceil\frac{k+5}{5}\right\rceil-\frac43
\]
is strictly smaller than the other two slopes, and hence the packing-refined
bound is eventually the smallest. This proves the claim.
}
\end{proof}

\smallskip
\noindent
{
The threshold \(k\ge23\) in the second item is asymptotic. The finite
comparison in Figure~\ref{fig:best-bounds-C8Pn-k1-2-8-25-n4-20} shows the
same transition already in a visible range of parameters. In particular, for
\(k=22\), the sparse periodic bound \eqref{eq:ub-C8Pn-linear} is still smaller
from \(n=27\) onward. From \(k=23\) onward, however, in the residue classes
\(k\not\equiv1\pmod5\), the packing-refined bound
\eqref{eq:packing-refined-general} becomes asymptotically dominant.
}

\section{Conclusion}

\smallskip
\noindent
{
In this paper, we investigated the $[k]$-Roman domination number of cylindrical
graphs $C_m\Box P_n$. We revisited the general density-type lower bound and
showed that it can be strengthened to the strict inequality
\[
\gamma_{[kR]}(C_m\Box P_n)> \frac{(k+1)mn}{5}.
\]
Since $\gamma_{[kR]}(C_m\Box P_n)$ is an integer, this implies the integer
lower bound
\[
\gamma_{[kR]}(C_m\Box P_n)\ge
\left\lfloor \frac{(k+1)mn}{5}\right\rfloor +1.
\]

\smallskip
\noindent
A key point in the proof is the analysis of the extremal local configuration
\[
f(N[x])=k+1 \quad \text{for every } x\in V(C_m\Box P_n).
\]
We showed that this configuration can occur only for efficient graphs,
namely for \(C_m\Box P_1\) with \(m\equiv0\pmod3\), and for
\(C_m\Box P_2\) with \(m\equiv0\pmod4\). Thus, efficient domination describes
exactly when the ideal local lower-bound configuration is attainable.

\smallskip
\noindent
It remains an interesting open problem to determine for which triples
\((k,m,n)\) the integer lower bound
\[
\left\lfloor \frac{(k+1)mn}{5}\right\rfloor +1
\]
is sharp. More generally, one may ask for the best possible additive
improvement
\[
\gamma_{[kR]}(C_m\Box P_n)\ge
\left\lfloor \frac{(k+1)mn}{5}\right\rfloor +1+\varepsilon_k(m,n),
\]
where $\varepsilon_k(m,n)\ge 0$ is as large as possible.

\smallskip
\noindent
For small fixed circumferences \(m\in\{5,\ldots,8\}\), we constructed explicit
periodic labelings yielding linear upper bounds, and refined them via uniform
ceiling-type constructions and packing-based reductions. These results could
naturally extend to multiples of the base circumferences by block repetition
along the cycle.

\smallskip
\noindent
Our approach highlights a structural interplay between three key concepts:
$[k]$-Roman domination, packing number, and efficient domination. The lower
bound is governed by local extremal constraints related to efficient
domination, while improvements of the upper bounds are achieved through
reductions on packing sets. This unified perspective clarifies how local and
global structural properties influence the behavior of $\gamma_{[kR]}$ on
cylindrical grids.

\smallskip
\noindent
Further research may focus on determining exact values for additional values
of \(m\), analyzing the tightness of the obtained bounds, and finding optimal
additive refinements of the density-type lower bound.
}

\smallskip
\noindent
{Inspecting the explicit periodic labelings constructed above, we observe that
the resulting upper bounds are linear in $n$ (up to an additive $O_k(1)$ boundary
term), and that their slopes are determined by the average density of positive
labels in the periodic patterns.

\smallskip
\noindent
More precisely, for $m=5,\ldots,8$ the leading terms are
\[
(k+1),\quad
\frac{4}{3}(k+1),\quad
\frac{3}{2}(k+1),\quad
2(k+1)-\frac{2}{5}.
\]

\smallskip
\noindent
On the other hand, the ceiling-type constructions yield slopes of the form
\[
m\left\lceil\frac{k+4}{5}\right\rceil
\quad\text{and}\quad
m\left\lceil\frac{k+5}{5}\right\rceil-\delta_m,
\]
where $\delta_m$ depends only on the packing density in $C_m\Box P_n$
and is independent of $k$.

\smallskip
\noindent
We see that the
resulting linear bounds are asymptotically linear in $n$, with slopes (for fixed
$k$) of the form $c_m(k+1)$ up to lower-order terms. For the cases $m=5,\dots,8$
treated in this paper, the leading coefficients are
\[
c_5=1,\qquad c_6=\frac{4}{3},\qquad c_7=\frac{3}{2},\qquad c_8=2.
\]
On the other hand, the ceiling-type constructions have asymptotic slope
$\frac{m}{5}k$ as $k\to\infty$. Comparing these slopes, we obtain
\[
\frac{m}{5}<c_m \quad\text{for } m=6,7,8.
\]

\smallskip
\noindent
Thus, in all cases considered here except $m=5$, the ceiling-type constructions
are asymptotically more efficient in the $n$-direction for sufficiently large
$k$. It would be interesting to determine whether the strict inequality
$\frac{m}{5}<c_m$ persists for all larger values of $m$, that is, whether for
every fixed $m\ge 6$ and sufficiently large $k$ the best ceiling-type (or
packing-improved) constructions eventually dominate the linear ones. Establishing
such a general comparison remains an open problem.}

\section*{Funding}
The first author (S.B.) acknowledges the financial support from the Slovenian Research Agency (ARIS)
through research programme No.\ P1-0297.
The second author (J.Z.) was partially supported by ARIS through the annual work program of Rudolfovo
and by the research grants P2-0248, L1-60136.
Also supported in part by Horizon Europe project Quantum Excellence Centre for Quantum-Enhanced Applications, QEC4QEA.

\section*{Declarations}
\noindent\textbf{Conflict of Interest.} The authors declare no conflicts of interest.

\bibliographystyle{unsrtnat}
\bibliography{references}
\end{document}